\documentclass{amsart}

\usepackage{verbatim}
\usepackage{graphicx}
\usepackage{amsfonts,amsmath,latexsym,amssymb,amsthm}
\usepackage[T1]{fontenc} 
\usepackage{geometry}
\newcounter{lemmacounter}

\newcounter{dummycounter}
\newcounter{propcounter}

\newcounter{emptycounter}

\newtheorem{lemma}[lemmacounter]{Lemma}

\newtheorem{proposition}[propcounter]{Proposition}

\newcounter{eqncounter}

\numberwithin{equation}{eqncounter}

\newcommand{\Oseen}{\mathcal{O}}

\newcommand{\IF}{\mathbb{F}}
\newcommand{\IC}{\mathbb{C}}
\newcommand{\IZ}{\mathbb{Z}}
\newcommand{\IN}{\mathbb{N}}
\newcommand{\IQ}{\mathbb{Q}}

\newcommand{\ord}{\text{\rm ord}}

\newcommand{\D}{D_K}
\newcommand{\del}{\eta(K)}

\newcommand{\aaa}{\mathfrak{a}}

\DeclareMathOperator{\Cl}{Cl}

\usepackage{color}
\definecolor{dblackcolor}{rgb}{0.0,0.0,0.0}
\definecolor{dbluecolor}{rgb}{0.01,0.02,0.7}
\definecolor{dgreencolor}{rgb}{0.2,0.4,0.0}
\definecolor{dgraycolor}{rgb}{0.30,0.3,0.30}

\title{A short remark on the $\ell$-torsion part of class groups}

\author{Martin Widmer}
\address{TU Graz\\
Institute of Analysis and Number Theory\\
Steyrergasse 30/II, 8010 Graz\\
Austria}
\email{martin.widmer@tugraz.at}

\begin{document}



\maketitle

\begin{abstract}
In a 2008 paper Ellenberg suggested a  strategy to improve the known upper bounds for the $\ell$-torsion part of class groups of number fields of fixed degree $d$. Motivated by this he proposed a question about the number of primitive elements of small height in a number field. Here we answer Ellenberg's question. 
We also improve Heath-Brown's bound for the $\ell$-torsion part
of class groups of purely cubic number fields, and we generalize our improvement to  pure fields of arbitrary odd degree $d$.  
\end{abstract}

\section{Introduction}
Let $K$ be a number field of degree $d>1$ and denote by  $\D$ the absolute value of its discriminant. 
For  arbitrary $\ell\in\IN=\{1,2,3,\ldots\}$  let $\Cl_K[\ell]=\{[\aaa]\in\Cl_K\: ;\: [\aaa]^\ell=[\Oseen_K]\}$ be the $\ell$-torsion of the ideal class group of the number field $K$.
The most important tool to upper bound  $\#\Cl_K[\ell]$ in terms of $\D,d,\ell$ is Ellenberg and Venkatesh's ``Key-Lemma'' \cite[Lemma 2.3]{EllVentorclass} which yields
\begin{alignat}1\label{ineq:EV} 
 \#\Cl_K[\ell] \ll_{d,\ell,\varepsilon} \D^{1/2-1/(2\ell(d-1))+\varepsilon},
 \end{alignat}
provided there are sufficiently many small suitable (e.g., splitting) primes, which is guaranteed under GRH. In a subsequent paper Ellenberg \cite{Elllowheight} pointed out that
the proof actually yields the stronger conclusion 
$$\#\Cl_K[\ell] \ll_{d,\ell,\varepsilon} \D^{1/2-f(\ell,d)+\varepsilon}$$
for a certain function $f(\ell,d)\geq 1/(2\ell(d-1))$ defined in (\ref{def:f}). Ellenberg proposes to study this function, 
more specifically he wrote:
{\it ``But at present it is not at all clear how to bound $M(K,\ell)$ or $f(\ell, d)$. As far as we
know it might be possible for $N_K'(X)$ to start growing quite quickly once it becomes
nonzero; in this case we would have $f(\ell,d) = 1/
2\ell(d-1)$ and no improvement would
be made on the results of \cite{EllVentorclass}. In fact, Lecoanet \cite{Lecoanet} has carried out experiments
for several dozen cubic fields $K$ which seem to show just this kind of behavior. It
would be very interesting to understand more fully the situation for cubic fields.''}\\

In this short note we compute  $f(\ell,3)$ and more generally  $f(\ell,d)$ whenever $\ell\geq d/2$. 

\begin{proposition}\label{prop:fdl}
Let $d>1$ and $\ell\geq d/2$ be integers. Then $f(\ell,d)=1/(2\ell(d-1))$.
\end{proposition}

Therefore, at least for $\ell\geq d/2$, a direct implementation of Ellenberg's idea does not work.
However, the topic has much evolved since 2008. The Key-Lemma has seen various refinements 
(\cite[Proposition 2.1]{Heath-BrownPierce}, \cite[Proposition 2.1]{doi:10.1007/s00208-020-02121-2}, and
\cite[Theorem 3.3]{KoymansThorner24}), leading to more suitable quantities than the above $N_K'(X)$.
For some of these refined quantities  sufficiently good upper bounds can be established, so that 
an adapted version of Ellenberg's idea does work, see, e.g., Chan and Koymans' new bound \cite[Theorem 1.1]{ChanKoymans25} on the $3$-torsion part of the class group of quadratic number fields.\\

There are rather few other instances where the hypothesis of the Key-Lemma can be established unconditionally, hence giving  unconditional pointwise upper bounds for the $\ell$-torsion.
Wang \cite{MR4237969} handled non-cyclic elementary abelian
extensions (and later extended her result to include many additional nilpotent extensions \cite{Wang2020}).  
Heath-Brown \cite[Theorem 10 and 11]{HB25} established (\ref{ineq:EV}) for quadratic and cubic fields with smooth discriminant. 
Assuming the minimal polynomial of a generator of $K/\IQ$ has a particular shape also allows for pointwise bounds.
Heath-Brown \cite[Theorem 1]{HB25} pointed out that (\ref{ineq:EV}) holds for pure cubic number fields, i.e., 
one has the unconditional upper bound 
\begin{alignat}1\label{ineq:HB}
\#\Cl_K[\ell]\ll_{\varepsilon,\ell}D_K^{1/2-1/(4\ell)+\varepsilon}
\end{alignat}
for $\ell$ prime and every $K=\IQ(\sqrt[3]{a})$ where $a>1$ is a cube-free integer. 
Using ideas of Dubickas we can sharpen (\ref{ineq:HB}).
\begin{proposition}\label{prop:HBD}
Let $K=\IQ(a^{1/d})$ be a pure cubic field, and $a=A_1A_2^2$ with positive integers $A_1,A_2$ both squarefree and coprime. 
Then we have
\begin{alignat*}1
\#\Cl_K[\ell]\ll_{\varepsilon,\ell}D_K^{1/2-1/(3\ell)+\varepsilon}A_2^{1/(3\ell)}
\end{alignat*}
for any $\varepsilon>0$, and any positive integer $\ell$.
\end{proposition}
We have $(A_1A_2)^2\ll D_K\ll (A_1A_2)^2$ (see (\ref{ineq:Disc}) below). By replacing $a$ with $a^2/A_2^3$ we can swap the roles of $A_1$ and $A_2$. Hence we can assume $A_2\leq A_1$, and therefore $A_2\ll D_K^{1/4}$.
If $a$ is squarefree (or squarefull) then we get
\begin{alignat*}1
\#\Cl_K[\ell]\ll_{\varepsilon,\ell}D_K^{1/2-1/(3\ell)+\varepsilon}.
\end{alignat*}
In the worst case when $A_1\asymp A_2$ we recover (\ref{ineq:HB}). It should also be mentioned that for $\ell=3$ 
one has
\begin{alignat*}1
\#\Cl_K[3]\ll_{\varepsilon}D_K^{\varepsilon}
\end{alignat*}
by work of Gerth \cite{Gerth76} (see \cite[(1.4)]{HB25} for more details). 

Our results apply for pure fields of any odd degree. 
We say $K$ is a pure field of degree $d$ if 
$K=\IQ(\theta)$, where $\theta$ is a root of some irreducible $f(x)\in \IZ[x]$ of the form $f(x)=x^d-a$.
Although stated only for cubic fields, Heath-Brown's observation  holds 
for pure fields of any odd degree, giving 
 \begin{alignat}1\label{ineq:SilHB}
\#\Cl_K[\ell]\ll_{\varepsilon,\ell,d}D_K^{1/2-1/(2(d-1)\ell)+\varepsilon}
\end{alignat}
for any $\varepsilon>0$, and any positive integer $\ell$ (not just primes).

If $K=\IQ(a^{1/d})$ is a pure field of degree $d$ then we can assume 
\begin{alignat}1\label{eq:a}
a=\prod_{i=1}^{d-1}A_i^i
\end{alignat} 
where the positive integers $A_1,\ldots,A_{d-1}$ are squarefree and pairwise coprime. 
If $k$ and $d$ are coprime then $(a^k)^{1/d}$ is also a generator of $\IQ(a^{1/d})$.
Therefore, one can replace\footnote{Suppose $k$ is coprime to $d$. Delete all $d$-th powers from $a^k$ to get $\phi_k(a)$. Then $\phi_k(a)^{1/d}$ is also a generator of $K$ and we have
\begin{alignat*}1
\phi_k(a)=\prod_{i=1}^{d-1}A_i^{[ik]},
\end{alignat*} 
where $[ik]\in \{1,2,\ldots,d\}$ with $ik\equiv [ik] \pmod d$.}
$A_1$ by $A_k$ as we did above 
whenever helpful.

It is clear that each prime $p$ dividing $A_1$
must occur with exponent at least $d-1$ in $\D$. 
This shows that $\prod_{(k,d)=1} A_k^{d-1}$ divides $\D$.

Each prime $p>d$ divides $\D$ at most with exponent $d-1$. For primes $p\leq d$
we note that the discriminant of the polynomial $x^d-a$ is $\pm d^da^{d-1}$.
Hence, we get
\begin{alignat}1\label{ineq:Disc}
\left(\prod_{(k,d)=1} A_k\right)^{d-1}\leq D_K \ll_d (A_1\cdots A_{d-1})^{d-1}.
\end{alignat}

Proposition \ref{prop:HBD} is an immediate consequence of the following  result.
\begin{proposition}\label{prop:GB}
Let $d>1$ be odd, and let $K=\IQ(a^{1/d})$ be a pure field of degree $d$. 
Then we have
\begin{alignat*}1
\#\Cl_K[\ell]\ll_{\varepsilon,\ell,d}D_K^{1/2+\varepsilon}\left(\min_{\frac{d+1}{2}\leq m\leq d-1}\prod_{i=1}^{d-1}A_i^{\frac{i\cdot m}{d}-\lfloor \frac{i\cdot m}{d}\rfloor}\right)^{-1/\ell}
\end{alignat*}
for any $\varepsilon>0$, and any positive integer $\ell$.
\end{proposition}
Let us consider two examples. If $a$ is squarefree (i.e., $a=A_1$) then we get  
\begin{alignat*}1
\#\Cl_K[\ell]\ll_{\varepsilon,\ell,d}D_K^{1/2-1/(2(d-1)\ell)-1/(2d(d-1)\ell)+\varepsilon},
\end{alignat*}
giving a small but significant improvement upon (\ref{ineq:SilHB}). If $a$ is cubefree (i.e., $a=A_1A_2^2$), then  we get  
\begin{alignat*}1
\#\Cl_K[\ell]\ll_{\varepsilon,\ell,d}D_K^{1/2-(d+1)/(2(d-1)\ell)+\varepsilon}A_2^{\frac{d-1}{2d\ell}},
\end{alignat*}
or equivalently
\begin{alignat*}1
\#\Cl_K[\ell]\ll_{\varepsilon,\ell,d}D_K^{1/2-1/(2(d-1)\ell)+\varepsilon}\left(\frac{A_2^{d-2}}{A_1}\right)^{1/(2d\ell)},
\end{alignat*}
giving an improvement upon (\ref{ineq:SilHB}) provided $A_1$ is sufficiantly large in terms of $A_2$ and $d$.

Throughout this article the implied constants in Vinogradov's notation $\ll$ and $\gg$ are absolute unless dependence on further parameters is explicitly mentioned by adding subscripts.    

\section{Background and definitions} 

Let 
\begin{alignat*}1
H_K(\alpha)=\prod_{v\in M_K}\max\{1,|\alpha|_v\}^{d_v}
\end{alignat*}
be the relative multiplicative Weil height of $\alpha\in K$.
Here $M_K$ denotes the set of places of $K$, and for each place $v$ we choose the unique representative $| \cdot |_v$ 
that either extends the usual Archimedean absolute value on $\IQ$ or a usual $p$-adic absolute value on $\IQ$, and $d_v = [K_v : \IQ_v]$ denotes the local degree at $v$.
Note that this is exactly the height  in \cite[(2.2)]{EllVentorclass} for the principal divisor $(\alpha, (\alpha))$ associated to $\alpha\in K^\times$.

Since $\max\{1,|\alpha\beta|_v\}\leq \max\{1,|\alpha|_v\}\max\{1,|\beta|_v\}$ we conclude that 
$$H_K(\alpha\beta)\leq H_K(\alpha)H_K(\beta)$$
for all $\alpha,\beta \in K$.
We write $\del$ for the minimal height of a primitive element\footnote{By Northcott's Theorem such a minimal element exists.}
\begin{alignat*}1
\del=\inf\{H_K(\alpha);K=\IQ(\alpha)\}.
\end{alignat*}

First we recall Ellenberg and Venkatesh's Key-Lemma \cite[Lemma 2.3]{EllVentorclass}. 
Recall from \cite{EllVentorclass} that a prime ideal $\mathfrak{p}$ of $\Oseen_K$ is said to be an extension of a prime ideal 
from a subfield $K_0\subsetneq K$ if there exists a prime ideal $\mathfrak{p}_0$ of $\Oseen_{K_0}$ such that $\mathfrak{p}=\mathfrak{p_0}\Oseen_K$. For such a prime ideal $\mathfrak{p}$ the residue degree $f(\mathfrak{p}/p)$ is necessarily larger than $1$.

If $\mathfrak{p}$ and $\mathfrak{p}_0$ are non-zero prime ideals in $\Oseen_K$ and $\Oseen_{K_0}$ respectively and $\mathfrak{p}\mid \mathfrak{p}_0\Oseen_K$ 
then we say $\mathfrak{p}$ is unramified in $K/K_0$
if $\mathfrak{p}^2\nmid \mathfrak{p}_0\Oseen_K$.

\begin{lemma}[Ellenberg and Venkatesh]\label{keylemma}
Suppose $K$ is a number field of degree $d>1$,  $\delta<1/(2\ell(d-1))$, and $\varepsilon>0$.
Moreover, suppose $\mathfrak{p}_1,\ldots,\mathfrak{p}_M$ are $M$ prime ideals in $\Oseen_K$ of norm
$N(\mathfrak{p}_i)\leq \D^\delta$ that are unramified in $K/\IQ$ and are not extensions of prime ideals from any proper subfield of $K$. 
Then we have 
$$\#\Cl_K[\ell] \ll_{d,\ell,\gamma,\varepsilon} \D^{1/2+\varepsilon}M^{-1}.$$
\end{lemma}

Here the hypothesis $\delta<1/(2\ell(d-1))$ can be replaced  by $\delta<\gamma/\ell$ as long as $\del>\D^\gamma$. 
This fact will be used for the proof of Proposition \ref{prop:HBD}.

It turns out that
for a ``typical'' $K$ one expects $\del$ to be much larger than  $\D^{1/(2(d-1))}$ and this led to improvements for the average
 $\#\Cl_K[\ell]$ over  various families (see \cite{doi:10.1112/blms.12113,elltor1,doi:10.1007/s00208-020-02121-2}).

Writing  $N'_K(X)$ for the number of primitive elements in $K$ of 
(relative) height less than $X$, Ellenberg \cite[Proposition 1]{Elllowheight} pointed out that the proof of \cite[Lemma 2.3]{EllVentorclass} even 
provides the stronger conclusion 
\begin{alignat}1\label{ineq:EVstrong}
\#\Cl_K[\ell] \ll_{d,\ell,\varepsilon} \D^{1/2+\varepsilon}X^{-1/\ell+\varepsilon}(1+N'_K(X)),
\end{alignat}
provided there are $\gg_{d,\ell,\varepsilon} X^{1/\ell-\varepsilon}$ prime ideals  $\mathfrak{p}$ in $\Oseen_K$ of norm
$N(\mathfrak{p})< X^{1/\ell}$ that are unramified in $K/\IQ$ and are not extensions of prime ideals from any proper subfield of $K$.
Following Ellenberg we define
$$M_{K,\ell}:=\inf_X(X^{-1/\ell}(1+N'_K(X))).$$
It is well-known (cf. \cite[Theorem 2]{Silverman} or \cite[Lemma 2]{8}) that $\del>(1/2)\D^{1/(2(d-1))}$, so that $N'_K(X)=0$ whenever $X\leq (1/2)\D^{1/(2(d-1))}$. Hence, 
$$M_{K,\ell}=\inf_{X\geq (1/2)\D^{1/(2(d-1))}}(X^{-1/\ell}(1+N'_K(X)))\leq 2\D^{-1/(2\ell(d-1))},$$ 
and thus 
\begin{alignat}1\label{def:f}
f(\ell,d):=\liminf_{K\atop [K:\IQ]=d}\frac{-\log M_{K,\ell}}{\log \D}\geq \frac{1}{2\ell(d-1)}.
\end{alignat}


\section{A lower bound for $M_{K,\ell}$}
To prove Proposition \ref{prop:fdl} we need a sufficently good lower bound for
$M_{K,\ell}$. And for this in turn we need to have a good lower bound on the number of small generators
$$N'_K(X)=\#\{\gamma\in K; K=\IQ(\gamma), H_K(\gamma)<X\}.$$
Surprisingly, it suffices to consider rational multiples of a minimal generator, and this gives us the next lemma. 
\begin{lemma}\label{lemma:count}
Let $K$ be a number field of degree $d>1$, and let $\delta>0$. Then $N'_K(\del \D^\delta)\gg \D^{2\delta/d}$.
\end{lemma}
\begin{proof}
Let $\alpha\in K$ with $H_K(\alpha)=\del$ and $\IQ(\alpha)=K$.     Now consider the primitive elements $\alpha\beta$ with nonzero $\beta\in \IQ$ and $H_K(\beta)< \D^{\delta/d}$.
Note that $H_K(\beta)=H_{\IQ}(\beta)^d$. Writing $\beta=b_1/b_0$ with coprime integers $b_0>0$ and $b_1$ it follows from the product formula that $H_{\IQ}(\beta)=\max\{b_0,|b_1|\}$.
Hence, there are $\gg T^2$ non-zero elements $\beta\in \IQ$ with $H_{\IQ}(\beta)<T$ if $T>1$. 
Applying this with $T=\D^{\delta/d}>1$ we conclude that there are $\gg \D^{2\delta/d}$ of these elements $\beta\in \IQ$. Using the trivial estimate $H_K(\alpha\beta)\leq H_K(\alpha)H_K(\beta)$ 
we conclude that $N'_K(\del \D^\delta)\gg \D^{2\delta/d}$.
\end{proof}
Now we can prove a lower bound for $M_{K,\ell}$.
\begin{lemma}\label{prop:MKl}
Let $K$ be a number field of degree $d>1$ and suppose $\ell\geq d/2$. Then
$$M_{K,\ell}\gg \del^{-1/\ell}.$$
\end{lemma}
\begin{proof}
If $X\leq \del$ then $N'_K(X)=0$. Hence,
\begin{alignat}1\label{term:MKX1}
\inf_{X\leq \del}(X^{-1/\ell}(1+N'_K(X)))=\del^{-1/\ell}.
\end{alignat}
If $X>\del$ then we can write $X=\del\D^\delta$ for some $\delta>0$, and thus 
\begin{alignat}1\label{term:MKX2}
\inf_{X> \del}(X^{-1/\ell}(1+N'_K(X)))=\inf_{\delta>0}\left((\del\D^\delta)^{-1/\ell}(1+N'_K(\del\D^\delta))\right).
\end{alignat}

Plugging the bound from Lemma \ref{lemma:count} into (\ref{term:MKX2}) shows that

\begin{alignat*}1
\inf_{X> \del}(X^{-1/\ell}(1+N'_K(X)))\gg \inf_{\delta>0}((\del\D^\delta)^{-1/\ell}\D^{2\delta/d})=\del^{-1/\ell}\inf_{\delta>0}\D^{\delta(2/d-1/\ell)}.
\end{alignat*}
Since $\ell \geq d/2$ we have $\inf_{\delta>0}\D^{\delta(2/d-1/\ell)}=1$, and thus
\begin{alignat}1\label{term:MKX3}
\inf_{X> \del}(X^{-1/\ell}(1+N'_K(X)))\gg \del^{-1/\ell}.
\end{alignat}
Combining (\ref{term:MKX1}) and (\ref{term:MKX3}) proves Lemma \ref{prop:MKl}.
\end{proof}

\section{proof of Proposition \ref{prop:fdl}}\label{sec:propfdl}
We have already seen that $f(\ell,d)\geq 1/(2\ell(d-1))$. Let us now suppose that $\ell\leq d/2$.
From Proposition \ref{prop:MKl} we know that $-\log M_{K,\ell}\leq \frac{1}{\ell}\log \del -\log C$ for some absolute constant $C>0$.
Hence,
\begin{alignat}1\label{term:proofPropfld1}
f(\ell,d)=\liminf_{K\atop [K:\IQ]=d}\frac{-\log M_{K,\ell}}{\log \D}\leq \liminf_{K\atop [K:\IQ]=d}\frac{\log \del}{\ell \log \D}.
\end{alignat}
The family of degree $d$-fields $K=\IQ(a^{1/d})$ as in (\ref{ineq:Disc}) with $a=A_1 A_{d-1}^{d-1}$ and $A_{d-1}\leq A_1\leq 2 A_{d-1}$ 
have a generator $\alpha=(A_1/A_{d-1})^{1/d}$ of height 
$$\del \leq H_K(\alpha)=A_1\leq \sqrt{2A_1A_{d-1}}\leq \sqrt{2}\D^{1/(2(d-1))}.$$

Plugging the above estimate  into (\ref{term:proofPropfld1}) 
for this infinite family  of fields  
shows that
$f(\ell,d)\leq 1/(2\ell(d-1))$. This completes  the proof of Proposition \ref{prop:fdl}.

\section{Constructing suitable primes}

The following lemma was observed by Heath-Brown (and possibly others including Jiuya Wang) but stated only for $d=3$.
For the convenience of the reader we give all details for general odd $d$.  
\begin{lemma}\label{lem:manygoodprimes}
Let $\delta>\varepsilon>0$. Then there are $\gg_{\varepsilon,d} D_K^{\delta-\varepsilon}$ many prime ideals $\mathfrak{p}|p$ in $\Oseen_K$ of degree $f(\mathfrak{p}/p)=1$,
ramification index $e(\mathfrak{p}/p)=1$, and norm $N(\mathfrak{p})<D_K^{\delta}$.
\end{lemma}
\begin{proof}
Let $K$ be a pure field of odd degree $d$. Hence there is a non-zero integer $m$, free of $d$-th powers, such that $f(x)=x^d-m$ is irreducible in $\IZ[x]$,  $K=\IQ(\theta)$ and $f(\theta)=0$. 

Note that the discriminant of $f$ has modulus $|\Delta_f|=d^dm^{d-1}$. 
Let $p|m$ be a prime divisor and $p^a$ the maximal power dividing $m$. Thus $1\leq a\leq d-1$. 
Let $\mathfrak{p}\subset \Oseen_K$ be a prime ideal above $p$. Since $m=\theta^d$ it follows that
$\mathfrak{p}|(\theta)$ and therefore $\mathfrak{p}^d|p^a$. Hence, $d\leq a\cdot e(\mathfrak{p}/p)$. This shows that $p$ ramifies in $K$.

The total number of primes $p$ that ramify in $K$ is $\omega(D_K)\ll_\varepsilon D_K^\varepsilon$. Next we show that for any unramified prime $p\equiv 2 \pmod d$
there exists a prime ideal $\mathfrak{p}|p$ in $\Oseen_K$ of degree $f(\mathfrak{p}/p)=1$, so that the claim
follows from Dirichlet's prime number theorem.   

So let $p$ be a prime with $p\equiv 2 \pmod d$ and $p\nmid m$. Let $g$ be a generator of the cyclic group $\IF_p^\times$. Reducing the polynomial 
$f$ modulo $p$ gives $\overline{f}(x)=x^d-g^t$ for some integer $t$. Hence $\overline{f}$ has a root $g^s$ in $\IF_p^\times$ if and only if
$sd\equiv t \pmod {p-1}$.  The latter has a solution $s$ if and only if $(d,p-1)|t$ which is true as $p\equiv 2 \pmod d$. Noting that $p\nmid [\Oseen_K:\IZ[\theta]]$ and applying Dedekind's factorisation theorem 
completes the proof.   
\end{proof}

\section{A height lower bound for generators of pure fields and proof of Proposition \ref{prop:GB}}
Ellenberg and Venkatesh used Silverman's classical lower bound for generators $\alpha$ of $K$
\begin{alignat}1\label{ineq:Silverman}
H_K(\alpha)\gg_d D_K^{1/(2(d-1))}
\end{alignat}
to get (\ref{ineq:EV}).
Dubickas \cite[Theorem 1]{Dubickas23} has improved (\ref{ineq:Silverman}) for pure fields $K=\IQ(a^{1/d})$ of odd degree $d$,
provided $a$ is prime. Following his proof and implementing the obvious minor modifications provides a lower bound for general $a$
that improves  (\ref{ineq:Silverman}) in some new cases, e.g., when $a$ is squarefree.

For the next result let $K=\IQ(a^{1/d})$ be a pure field of odd degree $d\geq 3$ with $a>1$. Recall that we can assume 
$$a=\prod_{i=1}^{d-1}A_i^i$$ 
where the positive integers $A_1,\ldots,A_{d-1}$ are squarefree and pairwise coprime.

\begin{proposition}[Dubickas, 2023]\label{prop:Dub}
Suppose $\alpha\in K$ and $K=\IQ(\alpha)$. Then 
\begin{alignat}1\label{ineq:Dub}
H_K(\alpha)> C_d \min_{\frac{d+1}{2}\leq m\leq d-1}\left(\prod_{i=1}^{d-1}A_i^{\frac{i\cdot m}{d}-\lfloor \frac{i\cdot m}{d}\rfloor}\right).
\end{alignat}
One can take $C_d=d^{-(2d-1)}$
\end{proposition}
\begin{proof}
For the convenience of the reader we reproduce  Dubickas' proof \cite[4. Proof of Theorem 1]{Dubickas23} with the necessary slight modification in the final step. Let 
$$\alpha=b_0+b_1 a^{1/d}+\cdots +b_{m}a^{m/d}$$ where $m\in \{1,2,\ldots,d-1\}$, $b_0,\ldots, b_m\in \IQ$ and $b_m\neq 0$. Replacing $\alpha$ by
$\alpha^{-1}$ (which does not change the height), we can assume that 
$m\geq (d+1)/2$ (see line after (4.3) in \cite{Dubickas23}). Let $T$ be the leading coefficient of the minimal polynomial of $\alpha$ in $\IZ[x]$ so that for $1\leq j\leq d$ the conjugates $\alpha_j$ satisfy 
$$\alpha_j=\sum_{k=0}^m b_ka^{k/d}\zeta^{(j-1)k}$$
where $\zeta=e^{2\pi i/d}$.  
By \cite[Lemma 7]{Dubickas23} there are $X_1,\ldots,X_{m+1}\in F=\IQ(\zeta)$ with $d^mX_j\in \Oseen_F$, and
\begin{alignat}1\label{eq:Xi}
X_1\alpha_1+\cdots + X_{m+1}\alpha_{m+1}=b_ma^{m/d}
\end{alignat}
and moreover,
\begin{alignat}1\label{ineq:Xi}
|X_j|\leq \frac{1}{(2\sin(\pi/d))^m} \quad (1\leq j\leq m+1).
\end{alignat}
We consider the $X_j$ and the conjugates $\alpha_i$ as complex numbers and $|\cdot|$ denotes the standard absolute value on $\IC$.
Note that $T\alpha_j$ and $d^m X_j$ are algebraic integers and therefore also  
$$d^mT(X_1\alpha_1+\cdots + X_{m+1}\alpha_{m+1})=d^mTb_ma^{m/d}$$
is an algebraic integer. Now $d^mTb_m$ is a non-zero rational number,
say $D_0/D$ with $(D_0,D)=1$ and $D\geq 1$. 

Combining (\ref{eq:Xi}) and (\ref{ineq:Xi}) gives 
$$a^{m/d}/D\leq d^mT|b_m|a^{m/d}\leq \frac{(m+1)d^mT\max_{1\leq j\leq m+1}|\alpha_j|}{(2\sin(\pi/d))^m}.$$

Since the height $H_K(\alpha)$ is equal to the Mahler measure of the minimal polynomial of $\alpha$ (see \cite[Proposition 1.6.6]{BG})
we get
$$H_K(\alpha)=T\prod_{j=1}^d\max\{1,|\alpha_j|\}\geq T\max_{1\leq j\leq m+1}{|\alpha_j|}\geq\frac{(2\sin(\pi/d))^m}{(m+1)d^m}\cdot \frac{a^{m/d}}{D}
> C_d\cdot \frac{a^{m/d}}{D}.$$
Next we claim that 
$$D\leq \prod_{p|a}p^{\lfloor\frac{m\cdot \ord_p(a)}{d}\rfloor}.$$
Now  $D_0 a^{m/d}/D$ and  $a^{(d-m)/d}$ are both algebraic integers.
Therefore, also their product $aD_0/D$ is an algebraic integer. Since $D_0$ and $D$ are coprime it follows that each prime divisor $p$ of $D$ 
must also divide $a$. Taking $d$-th powers and recalling that $D_0a^{m/d}/D$ is an algebraic integer
implies $d \cdot \ord_p(D)\leq m\cdot \ord_p(a)$, proving the claim.

Finally, we note that for $p|A_i$ we have $\ord_p(A_i)=i$, and hence,
$$\frac{a^{m/d}}{D}\geq\prod_{p|a}p^{\frac{m\cdot \ord_p(a)}{d}-\lfloor\frac{m\cdot \ord_p(a)}{d}\rfloor}\geq \prod_{i=1}^{d-1}A_i^{\frac{i\cdot m}{d}-\lfloor \frac{i\cdot m}{d}\rfloor}.$$
Recalling that $(d+1)/2\leq m\leq d-1$ this completes the proof of the lemma.  
\end{proof}

Using (\ref{ineq:Silverman}) to lower bound  $\eta(K)$ and applying Lemma \ref{lem:manygoodprimes} yields (\ref{ineq:SilHB}).
To prove Proposition \ref{prop:HBD} we use (\ref{ineq:Dub}) instead of (\ref{ineq:Silverman}). We apply the stronger form of the Key-Lemma (Lemma \ref{keylemma}), using the invariant $\eta(K)$, 
so that we can replace the hypothesis $\delta<1/(2\ell(d-1))$ by $\delta<\gamma/\ell$ as long as $\del>\D^\gamma$. 

We set 
$$\mathcal{A}:=C_d \min_{\frac{d+1}{2}\leq m\leq d-1}\left(\prod_{i=1}^{d-1}A_i^{\frac{i\cdot m}{d}-\lfloor \frac{i\cdot m}{d}\rfloor}\right).$$
Define $\gamma$ by $\mathcal{A}=\D^\gamma$ so that $\del>\D^\gamma$ by Proposition \ref{prop:Dub}. We can assume $\gamma>1/(2(d-1))$ as we already have (\ref{ineq:SilHB}). Next let $0<\varepsilon<\gamma/\ell$ and set $\delta=\gamma/\ell-\varepsilon$. Applying the Key-Lemma (Lemma \ref{keylemma}) and using Lemma \ref{lem:manygoodprimes} gives
\begin{alignat*}1
\#\Cl_K[\ell]\ll_{\varepsilon,\ell,d}\D^{1/2+\varepsilon}\D^{-(\delta-\varepsilon)}=\D^{1/2+3\varepsilon}\mathcal{A}^{-1/\ell}.
\end{alignat*}
This concludes the proof.

\bibliographystyle{amsplain}
\bibliography{literature}

\providecommand{\bysame}{\leavevmode\hbox to3em{\hrulefill}\thinspace}
\providecommand{\MR}{\relax\ifhmode\unskip\space\fi MR }
\providecommand{\MRhref}[2]{%
  \href{http://www.ams.org/mathscinet-getitem?mr=#1}{#2}
}
\providecommand{\href}[2]{#2}
\begin{thebibliography}{10}

\bibitem{BG}
E.~Bombieri and W.~Gubler, \emph{{H}eights in {D}iophantine {G}eometry},
  Cambridge University Press, 2006.

\bibitem{ChanKoymans25}
S.~Chan and P.~Koymans, \emph{A new pointwise bound for $3$-torsion of class
  groups}, 2025.

\bibitem{Dubickas23}
Art\=uras Dubickas, \emph{Minimal {M}ahler measures for generators of some
  fields}, Rev. Mat. Iberoam. \textbf{39} (2023), no.~1, 269--282. \MR{4571605}

\bibitem{Elllowheight}
J.~Ellenberg, \emph{Points of low height on $\mathbb{P}^1$ over number fields
  and bounds for torsion in class groups}, Computational arithmetic geometry,
  Contemporary Mathematics, vol. 463, Amer. Math. Soc., Providence, RI, 2008,
  45--48.

\bibitem{EllVentorclass}
J.~Ellenberg and A.~Venkatesh, \emph{Reflection principles and bounds for class
  group torsion}, {Int. Math. Res. Not.} \textbf{no.1, Art. ID rnm002} (2007).

\bibitem{elltor1}
C.~Frei and M.~Widmer, \emph{Average bounds for the $\ell$-torsion in class
  groups of cyclic extensions}, Res. Number Theory \textbf{4:34} (2018).

\bibitem{doi:10.1007/s00208-020-02121-2}
\bysame, \emph{Average bounds and higher moments for the $\ell$-torsion in
  class groups}, Math. Ann. \textbf{379} (2021), 1205--1229.

\bibitem{Gerth76}
Frank Gerth, III, \emph{Ranks of 3-class groups of non-{G}alois cubic fields},
  Acta Arith. \textbf{30} (1976), no.~4, 307--322. \MR{422198}

\bibitem{Heath-BrownPierce}
D.~R. Heath-Brown and L.~B. Pierce, \emph{Averages and moments associated to
  class numbers of imaginary quadratic fields}, Compositio Math. \textbf{153
  (11)} (2017), 2287--2309.

\bibitem{HB25}
D.R. Heath-Brown, \emph{{$\ell$-Torsion in Class Groups via Dirichlet
  L-functions}}, arXiv:2412.07701v2 (2025).

\bibitem{KoymansThorner24}
P.~Koymans and J.~Thorner, \emph{Bounds for moments of {$\ell$}-torsion in
  class groups}, Math. Ann. \textbf{390} (2024), 3221--3237.

\bibitem{Lecoanet}
D.~Lecoanet, \emph{{Report on undergraduate research project: elements of low
  height in cubic fields}},  (2008).

\bibitem{8}
D.~Roy and J.~L. Thunder, \emph{A note on {S}iegel's lemma over number fields},
  Monatsh. Math. \textbf{120} (1995), 307--318.

\bibitem{Silverman}
J.~Silverman, \emph{Lower bounds for height functions}, Duke Math. J.
  \textbf{51} (1984), 395--403.

\bibitem{Wang2020}
J.~Wang, \emph{Pointwise bound for $\ell$-torsion in class groups {II}:
  Nilpotent extensions}, 2020.

\bibitem{MR4237969}
\bysame, \emph{Pointwise bound for {$\ell$}-torsion in class groups: elementary
  abelian extensions}, J. Reine Angew. Math. \textbf{773} (2021), 129--151.
  \MR{4237969}

\bibitem{doi:10.1112/blms.12113}
M.~Widmer, \emph{Bounds for the $\ell$-torsion in class groups}, Bull. London
  Math. Soc. \textbf{50} (2018), no.~1, 124--131.

\end{thebibliography}

\end{document}